\documentclass[a4paper,10pt]{article}

\usepackage{amsthm}
\usepackage{amsmath}
\usepackage{amsxtra}
\usepackage{amssymb}
\usepackage{thmtools}
\usepackage{hyperref}
\usepackage{mathrsfs}
\usepackage{enumitem}
\usepackage[retainorgcmds]{IEEEtrantools}

\declaretheorem[name=Definition,style=definition,qed=$\dashv$,
numberwithin=section]{dfn}
\declaretheorem[name=Theorem,style=plain,sibling=dfn]{tm}
\declaretheorem[name=Lemma,style=plain,sibling=dfn]{lem}
\declaretheorem[name=Claim,style=plain,numbered=no]{clm*}
\declaretheorem[name=Question,style=definition,sibling=dfn]{ques}
\newcommand{\om}{\omega}
\newcommand{\rest}{\upharpoonright}
\newcommand{\inter}{\cap}
\newcommand{\ZF}{\mathsf{ZF}}
\newcommand{\ZFR}{\mathsf{ZFR}}
\newcommand{\OR}{\mathrm{OR}}
\newcommand{\sats}{\models}
\newcommand{\Ll}{\mathcal{L}}
\newcommand{\crit}{\mathrm{crit}}
\newcommand{\sub}{\subseteq}
\newcommand{\HC}{\mathrm{HC}}
\newcommand{\com}{\circ}
\newcommand{\elem}{\preccurlyeq}
\begin{document}
\title{A weak reflection of Reinhardt  by super Reinhardt cardinals}
\author{Farmer Schlutzenberg\footnote{Funded by the Deutsche 
Forschungsgemeinschaft (DFG, German Research
Foundation) under Germany's Excellence Strategy EXC 2044-390685587,
Mathematics M\"unster: Dynamics-Geometry-Structure.}\\
farmer.schlutzenberg@gmail.com}

\maketitle
\begin{abstract}
We prove a weakened version of the reflection of Reinhardt cardinals by super 
Reinhardt cardinals: Let $M=(V^M,P)$ be a countable model of
 second order set theory $\ZF_2$ (with universe $V^M$
 and classes $P$) which models ``$\kappa$ is 
super Reinhardt''. We show that there are unboundedly many $\mu<\kappa$
such that there
  is $j$ such that $(V^M,j)$ models $\ZF(j)+$``$\mu$ is Reinhardt, as witnessed 
by $j$''. In particular, $j\rest X\in V^M$ for all $X\in V^M$ (but we allow 
$j\notin P$).
\end{abstract}

Bagaria, Koellner and Woodin ask in \cite[Question 2]{woodin_koellner_bagaria}
whether super Reinhardt cardinals reflect Reinhardt cardinals;
that is, whether the least Reinhardt cardinal must be ${<}$ the 
least super Reinhardt.
We will not manage to answer that question, but we will prove 
an approximation which seems worth recording.

Note that the question is stated in second order set theory $\ZF_2$ (see 
\cite[p.~287]{woodin_koellner_bagaria}).
Recall 
that $\kappa$
is a Reinhardt 
cardinal iff $\kappa$ is the critical point of an elementary embedding $j:V\to 
V$,
and is super Reinhardt iff for every $\alpha\in\OR$ there
is an elementary $j:V\to V$ with critical point $\kappa$
and $j(\kappa)\geq\alpha$.

When we write $M=(N,P)\sats\ZF_2$, we follow \cite{woodin_koellner_bagaria},
so $M$ has sets in $N$ and classes in $P$, and $N\sats\ZF$, etc.
So here the logic is actually first-order.
The basic example is that if $\delta$ inaccessible then 
$(V_\delta,V_{\delta+1})\sats\ZF_2$, but we are interested in countable models 
below.

 $\ZFR$ denotes the first-order theory over the 
language $\Ll_{\ZFR}=\{{=},{\in},\tilde{j}\}$, with the $\ZF(j)$ axioms
(that is, like usual $\ZF$, but  Collection and Separation for 
all formulas over $\Ll_{\ZFR}$),
and the scheme asserting that $\tilde{j}:V\to V$ is elementary (actually it 
suffices to use the single axiom asserting that $\tilde{j}$ is 
$\Sigma_1$-elementary).

\section{Some reflection}

\begin{tm}\label{tm}
Let
$M=(N,P)\sats\ZF_2+$``$\kappa$ is super Reinhardt'',
and suppose that $N,P$ are countable.
 Then there is a set $X\sub\kappa$, with $X\in N$
 and $X$ of size $\kappa$ in $N$, such that for each $\mu\in X$ there is $j$
 such that $(N,j)\sats\ZFR$ and $\crit(j)=\mu$.
\end{tm}
 So $j\rest V_\alpha^N\in N$ for each $\alpha<\OR^N$,
but it seems that maybe $j\notin P$, so the theorem is not enough to conclude 
that some $\mu<\kappa$ is Reinhardt in $M$.

We begin with a standard definition and a simple lemma.

\begin{dfn}
 Let $\delta$ be a limit ordinal and $j:V_\delta\to V_\delta$ be  elementary.
 Write $j^1=j$, and for $j^{n+1}=j^n(j^n)$ be the $(n+1)$th iterate of $j$,
 for $n<\om$. Here for $A\sub V_\delta$ we define
 \[ j(A)=\bigcup_{\alpha<\delta}j(A\inter V_\alpha). \qedhere\]
\end{dfn}

\begin{lem}\label{lem:eventually_fixed}
Assume $\ZF$ and let $\delta$ be a limit and $j:V_\delta\to V_\delta$ be 
elementary.
 Then for each $n\in[1,\om)$:\footnote{We remark that more is established in 
\cite{reinhardt_non-definability}.}
 \begin{enumerate}[label=--]
  \item $j^n:V_\delta\to V_\delta$ is elementary,
  \item $j^{n+1}=j(j^n)$,
  \item for each $x\in V_\delta$, if $j^n(x)=x$ then $j^{n+1}(x)=x$.
  \end{enumerate}
  Moreover, for each $\alpha<\delta$ there is $n<\om$ such that 
$j^n(\alpha)=\alpha$.
\end{lem}

\begin{proof}
The first 3 items are clear. Now suppose the ``moreover'' clause is false and 
let $\alpha$ be the least counterexample. Then $\alpha+\om<j(\alpha)<\delta$.
For each $\beta<\alpha$ there is some $n$ such that $j^n(\beta)=\beta$.
Let
\[ A_n=\{\beta<\alpha\bigm|j^n(\beta)=\beta\}.\]
Then $\left<A_n\right>_{n<\om}\in V_\delta$, and note 
$\alpha=\bigcup_{n<\om}A_n$. So by elementarity,
\[ j(\alpha)=j(\bigcup_{n<\om}A_n)=\bigcup_{n<\om}j(A_n), \]
and since $\alpha<j(\alpha)$, therefore $\alpha\in j(A_n)$ for some $n$. But 
note
\[ j(A_n)=\{\beta<j(\alpha)\bigm|j^{n+1}(\beta)=\beta\}, \]
so $j^{n+1}(\alpha)=\alpha$, a contradiction.
\end{proof}

\begin{proof}[Proof of Theorem \ref{tm}]
 Work for the moment inside $M=(N,P)$.
\begin{clm*}
 There are unboundedly many $\mu<\kappa$
 such that there is an elementary $\pi:V_\kappa\to V_\kappa$
 with $\crit(\pi)=\mu$.
\end{clm*}
\begin{proof}
 Let $j:N\to N$ with $\crit(j)=\kappa$. Let $\lambda>\kappa$
 with $j(\lambda)=\lambda$. 
 Let $k:N\to N$ with $\crit(k)=\kappa$ and $k(\kappa)>\lambda$.
 Let $\beta>k(\kappa)$ be such that $j(\beta)=\beta$.
 Let $\ell:V_\beta\to V_\beta$ be $\ell=j\rest V_\beta$.
 Then $\ell^n=j^n\rest V_\beta$ for each $n<\om$,
 and by Lemma \ref{lem:eventually_fixed}, we can fix $n<\om$ such that 
$\ell^n(k(\kappa))=k(\kappa)$.
Let $\pi'=\ell^n\rest 
V_{k(\kappa)}$. Then $\pi':V_{k(\kappa)}\to V_{k(\kappa)}$ and 
\[ \kappa\leq\crit(\pi')=\crit(\ell^n)=\crit(j^n)<\lambda<k(\kappa).
\]
The existence of $\pi'$ now reflects under $k$ to yield some $\pi$ as desired.
\end{proof}

Let $X$ be the set of all $\mu$ as in the claim. By the claim, $X$ has size 
$\kappa$.

Now step out of $M$, back to $V$, where $M\in\HC$.
We claim that $X$ is as desired.
So let $\mu\in X$ and fix  $\pi\in N$ witnessing this. Note that 
$(V_\kappa^N,\pi)\sats\ZFR$,
since $\kappa$ is inaccessible in $N$. Let $\kappa_0=\kappa$ 
and $\pi_0=\pi$ 
and $Q_0=(V_{\kappa_0}^N,\pi_0)$.
Using super Reinhardtness (which is easily preserved under elementary $k:N\to 
N$ when $k\in P$), construct a sequence 
$\left<j_n\right>_{n<\om}$ of embeddings $j_n:N\to 
N$ with $j_n\in P$, such that $\crit(j_0)=\kappa_0$
and 
\[ \kappa_{n+1}=\crit(j_{n+1})=j_n\com\ldots\com j_0(\kappa_0), \]
and with $\sup_{n<\om}\kappa_n=\OR^N$. Let
$Q_{n+1}=j_n\com\ldots\com j_0(Q_0)$.
Note then that $Q_n\sats\ZFR$ and $Q_n\elem Q_{n+1}$,
and $Q_n$ has universe $V^N_{\kappa_n}$.
So let
\[ Q_\om=(N,\pi_\om)=\bigcup_{n<\om}M_n.\]
Then $Q_\om$ is the union 
of an elementary chain, so $Q_\om\sats\ZFR$,  $Q_\om$ has universe 
$N$, and $\crit(\pi_\om)=\crit(\pi)=\mu$.
\end{proof}

Note that
because we had to iterate the embedding
$\ell$ in order to fix $k(\kappa)$, the proof leaves open the 
questions:

\begin{ques}
With notation as above, is there some $j:N\to N$ with 
$j\in P$ and $\crit(j)=\kappa$,
such that letting $D$ be the normal measure on $\kappa$
derived from $j$, there is a $D$-measure one set of
 $\mu$ as before?
 
 Can one prove the same conclusion but without assuming that $M$
 is countable?
 
 And of course, can one get $j\in P$ with $\crit(j)<\kappa$?
 \end{ques}

 \bibliographystyle{plain}
\bibliography{../bibliography/bibliography}

\end{document}